\documentclass[11pt]{amsart}

\usepackage{amsmath, amsthm, amssymb, amsfonts, amscd, latexsym}
\usepackage[alphabetic]{amsrefs}

\usepackage[T1]{fontenc}
\usepackage{dsfont}
\usepackage{mathtools}              % math tools
\usepackage{mathrsfs}
\usepackage{palatino}               % text font
\usepackage{comment}

\usepackage[abs]{overpic}		  % overpic automatically loads graphicx

\usepackage{xcolor}
\definecolor{indigo}{rgb}{0.29, 0.0, 0.51}  % custom colors
\usepackage[colorlinks, urlcolor=indigo, linkcolor=indigo, citecolor=indigo]{hyperref}
\usepackage{cleveref}

\usepackage[hcentering, top = 1.5in, total={6in, 8.3in}]{geometry}  % 1.35 vertical margin; 1.25 horizontal margin

\usepackage{tikz} %provides edge/point drawing
    \usetikzlibrary{arrows,shapes}
    \usetikzlibrary{fit,shapes.geometric} %provides tight ellipse around nodes
    \usetikzlibrary{decorations.pathreplacing,angles,quotes} %brackets    
\usepackage{tikz-cd}

\theoremstyle{plain}
\newtheorem{theorem}{Theorem}
\numberwithin{theorem}{section}

\newtheorem{proposition}[theorem]{Proposition}

\theoremstyle{definition}
\newtheorem*{acknowledgments}{Acknowledgments}

\theoremstyle{remark}
\newtheorem{remark}[theorem]{Remark}
\newtheorem{example}[theorem]{Example}

\newcommand{\RR}{\mathbb{R}}           % the real numbers\

\newcommand{\PP}{\mathbb{P}}
\newcommand{\TT}{\mathbb{T}}
\newcommand{\CP}{\mathbb{CP}}         % complex projective space
\makeatletter
\newcommand*\bigcdot{\mathpalette\bigcdot@{0.6}}
\newcommand*\bigcdot@[2]{\mathbin{\vcenter{\hbox{\scalebox{#2}{$\m@th#1\bullet$}}}}}
\makeatother

\begin{document}
%%%%%%%%%%%%%%%%%%%%%%%%%%%%%%%%%%%%%%%%%%
% title
\title[]{K\"{a}hler-Ricci Flow is Symplectic Deflation plus an Exact Symplectic Deformation}

% author information

\author{Randall R. Van Why}

\address{School of Mathematics \\ Georgia Institute of Technology \\  Atlanta, GA}
\email{randyvw@gatech.edu}

% abstract
\begin{abstract}
We demonstrate that K\"{a}hler-Ricci flow on a Fano surface decomposes as a symplectic deflation along a plurianticanonical divisor plus an exact symplectic deformation in the divisor complement. We apply this decomposition to the study of certain projective bundles. We establish the optimal decay rate in the setting of Song, Sz{\'e}kelyhidi, and Weinkove for a family of toric metrics which match the cohomological profile of K\"{a}hler-Ricci flow.  In doing so, we identify a $C^2$-obstruction to proving the optimal decay rate in the setting of ruled toric surfaces. We present these ideas as a new symplectic-topological approach to the study of K\"{a}hler-Ricci flow.
\end{abstract}

\maketitle

\tableofcontents

%\textbf{Your job as a Geometer is to elucidate the structure of the spirit of Geometry in the most beautifully meaningful way possible.}

\section{Introduction}
Suppose that $(X,g, \omega)$ is a K\"{a}hler manifold. The K\"{a}hler-Ricci flow equation \cite{Cao1985, SongWeinkove2013} is given by
\begin{equation}\label{eqn:Ricci_flow}
\begin{cases}
    \partial_t \omega_t = -Ric_{\omega_t}\\
    \omega_0 = \omega
\end{cases}
\end{equation}
Any 1-parameter family $(X,g_t,\omega_t)$ of K\"{a}hler manifolds satisfying this equation is called a \emph{solution} to the K\"{a}hler-Ricci flow equation. Here $Ric_{\omega_t}$ is the Ricci form of the K\"{a}hler metric $g_t$ defined by 
\[Ric_{\omega_t} = -i\partial \overline{\partial} \log\left(\det g_{j\overline{k}}(t)\right)\]
Cohomologically 
\[[Ric_{\omega_t}] = 2\pi c_1(X)\]
and so we have an induced flow on cohomology
\begin{equation}\label{eqn:cohomology}
\begin{cases}
    \partial_t [\omega_t] = -2\pi c_1(X)\\
    [\omega_0] = [\omega]
\end{cases}
\end{equation}
which has a unique solution $[\omega_t] = [\omega_0] - 2\pi t c_1(X)$. 

This observation leads one to conclude that for \emph{Fano} manifolds, manifolds with $c_1(X) > 0$, one should expect a finite extinction time for $X$ under the K\"{a}hler-Ricci flow. In 2003 through public communication during a visit to MIT (c.f. \cite{SesumTian2008}), Perelman demonstrated that for a Fano manifold $X$, under the assumption that $[\omega_0] = 2\pi Tc_1(X)$, one must have 
\begin{align}
    diam_{\omega_t}(X) \le C\sqrt{T-t}
\end{align} 
where the left hand side is the diameter of $X$ with respect to the K\"{a}hler metric $\omega_t$. Sesum and Tian in \cite{SesumTian2008} published a proof of Perelman's estimate. It was then conjectured by Tian \cite{Tian2008} and later proven by Song \cite{Song2014} that the condition $[\omega_0] = 2\pi Tc_1(X)$ holds if and only if $X$ is Fano, so one should always expect such a decay rate in that setting. This fact and other related results \cite{FeldmanIlmanenKnopf2003,TianZhang2006,Tian2008}  have greatly contributed to the present day analytic minimal model program which attempts to center K\"{a}hler-Ricci flow as a process that tends to converge toward minimal birational models through singularities such as analytic blowdowns \cite{SongWeinkove2013:contract, SongWeinkove2014}.

Perelman's diameter decay rate estimate has been shown to have a major impact on the strucuture of other complex manifolds under K\"{a}hler-Ricci flow. In \cite{SongWeinkove2011}, Song and Weinkove make advances toward systemetizing these Fano collapse results to Fano fibrations through their study of Hirzebruch surfaces. Subsequent work by \cite{Fong2014} on projective bundles continued this general direction. A major breakthrough came with \cite{SongSzekelyhidiWeinkove2013} in which they proved that a projective bundle $\pi \colon (X,\omega_0) \to (B,\omega_B)$ under the assumption that 
\[[\omega_0] - 2\pi Tc_1(X) = [\pi^*\omega_B]\]
admits the following estimate for the diameter of the fiber $F$
\[diam_{\omega_t}(F) \le C(T-t)^{1/3}.\]
In their work, it is noted that this falls short of the expected optimum $\sqrt{T - t}$ bound. Several improvements were made in various settings and this thinking was applied to other Fano fibrations as well \cite{FuZhang2017,Shen2021,CollinsTosatti2015,TosattiZhang2018}. In 2015, Fong achieved the optimal upper bound for the diameter of the fiber under the assumption of a Ricci curvature upper bound \cite{Fong2015}. For projective bundles, research developments resulted in Zhang-Zhang's 2023 improvement of the exponent to $1/2- \varepsilon$ for every $\varepsilon > 0$, just shy of the optimum 1/2 \cite{ZhangZhang2023}. The author also notes very recent results due to Bednarek \cite{Bednarek2025}, Jian-Song \cite{JianSong2026,JianSong2026b}, and Xu-Zheng \cite{XuZhang2026} which achieve the optimum decay rate in various settings.

In this article, we use symplectic geometry to improve our understanding of the optimal collapse rate in the settings of toric ruled surfaces. We use our findings in order to show the following theorem in that context (c.f. \cite{SongSzekelyhidiWeinkove2013}, Section 1).
\begin{theorem}\label{thm:main}
    Let $(X,g,\omega)$ be a toric ruled surface with compatible ruling. Let $(g_t,\omega_t)$ be the K\"{a}hler-Ricci flow on $X$ starting at $(g,\omega)$ such that the volume of the fiber $F$ collapses at time $T > 0$. Then there exists a 1-parameter family of toric K\"{a}hler metrics $\widetilde{\omega}_t$ such that 
    \[\omega_t = \widetilde{\omega}_t + i\partial \overline{\partial}\psi_t\]
    and we have the following estimate
    \[\text{diam}_{\widetilde{\omega}_t}(F) \le C(T - t)^{1/2}.\]
    This establishes the sharp upper bound in the setting of Song-Sz\'{e}kelyhidi-Weinkove in \cite{SongSzekelyhidiWeinkove2013} for the family $\widetilde{\omega}_t$.
\end{theorem}

Our proof of this result exploits the rich symplectic-topological rigidity near an anticanonical divisor in a Fano manifold. We show that K\"{a}hler-Ricci flow on a Fano manifold decomposes into two purely symplectic processes: symplectic deflation along a (pluri)-anticanonical divisor and an exact symplectic deformation in the divisor complement $X - D$. We view  \Cref{thm:main} result as a natural toric realization of this decomposition. In light of \cite{SongSzekelyhidiWeinkove2013} and others, we discuss how $\psi_t$ is a $C^2$-obstruction to the optimal bound for the decay rate of the fiber under $\omega_t$ and urge a better understanding of its regularity. We present our work as a new approach to the study of K\"{a}hler-Ricci flow.

\begin{acknowledgments}
Thank you to Gábor Székelyhidi and Valentino Tosatti for remarks and resources. We thank John Etnyre, Joshua Jordan, Jie Min, and Jeffery Streets for lending an ear through the initial ambiguity.  Special thanks to Hiro Lee Tanaka for what ended up being radical encouragement.
\end{acknowledgments}

\section{K\"{a}hler Ricci Flow Relative to a Divisor}\label{sec:relative}

%%%%%%%%%%%%%%%%%%%%%%%%%%%%%%%%%%%%%%%%%%
\subsection{Symplectic Deflation Near Plurianticanonical Divisors}
%%%%%%%%%%%%%%%%%%%%%%%%%%%%%%%%%%%%%%%%%%

The cohomological solution
\[[\omega_t] = [\omega_0] - 2\pi t c_1(X)\]
to the K\"{a}hler-Ricci flow
implies that the flow simply deforms the cohomology class of the K\"{a}hler form linearly along the first Chern class. Let $(X,g,\omega)$ be a Fano (or Calabi-Yau) surface. We say a divisor $D \subset X$ is \emph{plurianticanonical} if it is Poincare dual to a positive multiple of the first Chern class $c_1(X)$, write
\[PD[D] = kc_1(X).\]
Equivalently, $D$ finds itself in the same divisor class as the $k$-plurianticanonical bundle $D\in |-K_{X}^{\otimes k}|$. For now, we will assume that this divisor is smooth. The normal bundle 
\[\nu(D) \to \Sigma\] 
provides a symplectic model for a symplectic tubular neighborhood $N_D \subset X$. The Thom class of this bundle may be represented by a 2-form $\omega_{D} \in \Omega_{dR}^2(N_D) \xhookrightarrow{} \Omega_{dR}^2(M)$ compactly supported near the divisor. Consider the family of de Rham forms 
\[\omega_t = \omega - \frac{2\pi t}{k}\omega_{D}\] 
for $t \in [0,\varepsilon]$. This form is symplectic for $\varepsilon > 0$ sufficiently small. The 1-parameter family $(X,\omega_t)$ is a symplectic deformation known as an \emph{symplectic deflation along $D$}.
The class $c_1(-K_{X}^{\otimes k})$ restricts to the Thom class of the divisor $D$ on any tubular neighborhood of $D$. Via the representative $\omega_D$ of the Thom class, the cohomological solution becomes
\[[\omega_t] = [\omega_0] - \frac{2\pi t}{k}[\omega_D].\]
Cohomologically, we see that K\"{a}hler-Ricci flow appears to be simply performing symplectic deflation along $D$. 

\subsection{Volume evolution and finite collapse}

If $\Sigma \subset M$ is a 2-cycle, such as a $J$-holomorphic curve or some other divisor, then along the deformation we have 
\begin{align*}
Vol(\Sigma,\omega_t) &= \langle [\omega_t], [\Sigma] \rangle\\ 
&= \int_{\Sigma} \omega_t\\
&= \int_{\Sigma} \omega  - \frac{2\pi t}{k}\int_{\Sigma} \omega_D\\
&=  Vol(\Sigma, \omega) - \frac{2\pi t}{k} [\Sigma] \cdot [D] 
\end{align*}
so 
\[\frac{d}{d t} Vol(\Sigma,\omega_t) = -\frac{2\pi [\Sigma] \cdot [D]}{k}\]
and thus the change of the area of a 2-cycle $\Sigma$ along the deflation depends on the sign of its intersection with the divisor $D$ and its speed depends on the number of intersection points. This also applies to the divisor $D$ itself
\[\frac{d}{dt} Vol(D,\omega_t) = -\frac{2\pi [D]\cdot [D]}{k}\]
where $[D] \cdot [D]$ is the self-intersection number of the divisor (assumed to be a smooth surface). For this reason, positive divisors induce an upper bound on the existence time of each curve under K\"{a}hler-Ricci flow 
\[T_{\Sigma} = \frac{kVol(\Sigma,\omega)}{2\pi[\Sigma]\cdot[D]}.\]

We let $T_{max}$ be the maximum $\varepsilon$ along which $[\omega_t]$ remains positive along the deflation. Then we may bound this time by the value $T_D$ when $Vol(D,\omega_{T_D}) = 0$ so
\[T_{max} \le T_D = \frac{kVol(D,\omega)}{2\pi[D]\cdot[D]}.\]

Suppose now that $D$ has simple normal crossing singularities. One may still perform deflation just as easily. Suppose that a simple normal crossing (snc)-divisor $D$ decomposes into smooth components as 
\[D = \bigcup_{i=1}^N D_i.\] 
We let $\omega_{i}$ be a compactly supported 2-form representing the the Thom class associated to the normal bundle $\nu(D_i) \to D_i$. Then letting $\omega_D = \sum_{i = 1}^N \omega_i$, we may form $\omega_t = \omega - \frac{2\pi t}{k} \omega_D$ as before. Now we have 
\[\frac{d}{dt} Vol(D_i, \omega_t) = -\frac{2\pi[D_i] \cdot [D]}{k} = -\frac{2\pi(d_i + k_i)}{k}\]
where $d_i$ is the \emph{degree} of the edge of the divisor graph $\Gamma_D$ corresponding to the component $D_i \subset D$ and $k_i = [D_i] \cdot [D_i]$. Depending on the sign of this integer, the volume of a component may increase or decrease. If $d_i + k_i$ is positive, then the component collapses at time
\[T_{D_i} = \frac{kVol(D_i,\omega)}{2\pi(d_i + k_i)}\]
and we may refine our estimate
\[T_{max} \le \min_{i} T_{D_i} \le T_D.\]

\subsection{Deflation plus an exact deformation in the complement}

As we have seen, if $D \subset (X,g,\omega)$ is \emph{smooth} and plurianticanonical, we may choose the 2-form $\omega_D$ representing its Thom class to be supported in a symplectic tubular neighborhood $N_D$ of $D$ modeled on the normal bundle of $D$. Because of the cohomological identity 
\[[\omega_t] = [\omega] - \frac{2\pi t}{k}[\omega_D]\]
then we must have 
\[\int_{\Sigma} \omega_t - (\omega - \frac{2\pi t}{k}\omega_{D}) = 0\]
for every 2-cycle $\Sigma$, and so $\omega_t - (\omega - \frac{2\pi t}{k}\omega_D)$ is an exact 2-form. This allows us to write
\[\omega_t = \omega - \frac{2\pi t}{k}\omega_D + d\beta_t\]
for some 1-form $\beta_t \in \Omega_{dR}^1(X).$ Thus we see that K\"{a}hler-Ricci flow differs from actual symplectic deflation only by the exact form $d\beta_t$.

Since $Supp(\omega_D) \subset N_D$, in the complement $\overline{X} = X - N_D$, we have $\omega_D = 0$ and the deformation becomes 
\[\omega_t = \omega + d\beta_t\]
so K\"{a}hler-Ricci flow is a purely exact symplectic deformation there.

\begin{remark}
    One can generalize this setup to $N_D$ a regular symplectic neighborhood of a divisor with simple normal crossings. These neighborhoods are built from smoothings of plumbings of tubular neighborhoods of their smooth components. See for example \cite{LiMin2021}.
\end{remark}
These straightforward observations allow one to conclude
\begin{theorem}
Let $(X,g,\omega)$ be a Fano variety and let $D \subset X$ be a smooth (or snc) $k$-plurianticanonical divisor. Then under K\"{a}hler-Ricci flow $(X,g_t,\omega_t)$ decomposes into two simultaneous processes:
\[\omega_t = \omega^{Deflation}_t + \omega_t^{Exact}\]
\begin{enumerate}
\item \textbf{Deflation} along the symplectic divisor $D$ supported in a tubular neighborhood $N_D \subset X$ with  
\[\omega^{Deflation}_t = \omega - \frac{2\pi t}{k}\omega_D\] 
which is symplectic for short $t$
and\\
\item a global \textbf{exact symplectic deformation}
\[\omega^{Exact}_t = d\beta_t.\]
\end{enumerate}
\end{theorem}

This decomposition presents a global organizing framework for K\"{a}hler-Ricci flow in the Fano setting. This perspective can be further refined in case of toric surfaces and its consideration there will lead to our proof of \Cref{thm:main}.

\section{Toric Varieties}
We now specialize to a smooth proper toric surface $(X,g,\omega)$. Such a manifold admits a \emph{toric moment map} $\mu \colon X \to \RR^2$. The image of the moment map is a closed convex polytope $P \subset \RR^2$, its \emph{Delzant polytope} or \emph{moment polytope} \cite{Abreu2003,CLS2011}. The preimage of the interior of the polytope $U = \mu^{-1}(P-\partial P)$ is an open dense submanifold with a free $\TT^2$-action. The fiber $T_y = \mu^{-1}(y)$ over an interior point $y \in P$ is a torus which is invariant under the Hamiltonian $\TT^2$-action.

\subsection{Convex Toric Potentials}
A convex polytope can be described via a collection of linear functions $\{\ell_k\}_{k=1}^N$ defining its faces: Each face $F_k \subset \partial P$ of the polytope lies on a unique line which has a unique outward-pointing primitive normal vector $\vec{n}_k$. The Delzant polytope $P$ may be presented as an intersection of the half spaces defined by inequalities (c.f. \cite{Guillemin1994})
\begin{align*}
\langle y, \vec{n}_1 \rangle &\le \lambda_1\\
\langle y, \vec{n}_2 \rangle &\le \lambda_2\\
&\vdots\\
\langle y, \vec{n}_N \rangle &\le \lambda_N
\end{align*}
with $\lambda_k \in \RR$. We may then define the functions
\[\ell_k(y) = \lambda_k - \langle y, \vec{n}_k \rangle\]
so that a point $y$ is in the interior of $P$ if and only if $\ell_k(y) > 0$ for all $k$.

Following the framework of Abreu and Guillemin \cite{Abreu1998,Guillemin1994}, each K\"{a}hler potential $\phi \colon X \to \RR$ corresponds under the Lengendre transform to a convex function $u \colon U \to \RR$ over the interior of the moment polytope $P$. In toric action-angle coordinates $(y_1,y_2,\theta_1,\theta_2)$, there is a canonical choice of $u$ (c.f. \cite{Abreu1998,Guillemin1994})
\begin{align}\label{eqn:guillemin}
    u(y) = \frac{1}{2}\sum_{k=1}^N \ell_k(y)\log(\ell_k(y)).
\end{align}
By setting $x_i = \frac{\partial u}{\partial y_i}$, we may obtain a K\"{a}hler potential $\phi$ for a K\"{a}hler form $\omega$ in $(x_1,x_2)$-coordinates via Legendre transform \cite[Section 2.2]{Abreu2003}. We may exponentiate these to complex coordinates $z_i = e^{x_i + i\theta_i}$ for the interior $X - D$. 
We also have \cite[Exercise 3.5]{Abreu2003} \cite[Chapter 8]{CLS2011}
\[\frac{\partial \phi}{\partial x_i} = y_i.\]

\[[\omega] = 2\pi\sum_{k=1}^N \lambda_k[\omega_{D_k}]\]
and 
\[c_1(X) = \sum_{k=1}^N [\omega_{D_k}].\]
A quick computation shows

\begin{proposition}
    We have 
    \[du = \frac{1}{2}\sum_{k=1}^N [1 + \log(\ell_k(y))]d\ell_k\]
    \[\nabla^2 u = \frac{1}{2}\sum_{k=1}^N\frac{d\ell_k\otimes d \ell_k}{\ell_k(y)}\]
    for $u, \ell_k$ defined above.
\end{proposition}

\subsection{Toric Ruled Surfaces}

Let $f \colon X \to \Sigma$ be a holomorphic fibration. We say such a fibration is \emph{compatible} with a toric structure if its restriction to each fiber is toric. For toric surfaces, this implies that each fiber is a toric Riemann surface which then implies that it is diffeomorphic to $S^2$. Thus $f \colon X \to \Sigma$ is actually a $\PP^1$-bundle over $\Sigma$. This is sometimes referred to as a $T^2$-invariant ruling on $X$. We shall refer to $X$ as a \emph{toric ruled surface}.

\begin{example}
    Consider $X = \PP^1 \times \PP^1$ with standard toric moment map $\mu \colon X \to \RR^2$ given by the product of the standard moment map for $\PP^1$. The image of $\mu$ is a square $P = [0,1]^2 \subset \RR^2$. The holomorphic submersion $f \colon X \to \PP^1$ given by $(p_1,p_2) \mapsto p_1$ is compatible with the toric structure.
\end{example}

\begin{example}Take $\CP^2$ together with an embedded projective line $L \subset \CP^2$. We may fit this line to a Lefschetz pencil on $\CP^2$ whose base locus is a single point $p$. Blowing up this point $p$, we obtain a Lefschetz fibration $F \colon X \to \PP^1$. The toric structure on $\CP^2$ may be used to construct a toric structure on $X$. Via a transformation of $\CP^2$, we can ensure that our Lefschetz fibration is also compatible with this toric structure.
\end{example}

\begin{example}
    More generally, we consider a Hirzebruch surface $X_{b-a} = \mathcal{\PP}(\mathcal{O}(a)\oplus \mathcal{O}(b))$. This carries an effective $\mathbb{T}^2$-action which descends from the effective $\TT^2$-action on the rank-2 bundle $\mathcal{O}(a)\oplus \mathcal{O}(b)$.
\end{example}

Possessing a compatible Hamiltonian $S^1$ action on fibers leads to remarkable rigidity. For each fixed fiber, such an action represents a subgroup $S^1 \xhookrightarrow{} \TT^2$, a periodic 1-parameter family of elements. Such a family is characterized by its slope at the identity. This is represented by a vector $v \in \mathfrak{t}$ coming from some parameterization of $S^1$. In the standard parameterizaiton of $\TT^2$, this is a vector of some fixed rational slope. Such a holomorphic submersion thus represents a continuous family of rational slopes which only exists if the family is locally constant. It then follows that such an action may be represented by a non-vanishing assigment of parallel representing vectors $\sigma \colon \Sigma \to \mathfrak{t}$. By definition, each of these generates a vector field $X_{\sigma(q)}$ which is tangent to its corresponding fiber $f^{-1}(q)$. From such a vector, one can choose a single vector $v \in \mathfrak{t}$ which generates a vector field $X_v$ tangent to all fibers simultaneously (in fact every vector in the image of $\sigma$ satisfies this). This yields a potential condition to check for establishing compatibility for more specific examples \cite[Chapters 3 and 7]{CLS2011}.
\begin{proposition}
    A submersion $f \colon X \to \Sigma$ is compatible with the toric structure on $X$ if and only if there is a primitive vector $v \in \mathfrak{t}$ such that, for each fiber $F$, $X_v \in TF$ everywhere.
\end{proposition}

\subsection{A shrinking family of toric metrics}
From the presentation of the moment polytope $P$ via the collection of inequalities 
\[\langle y, n_k\rangle  \le \lambda_k\]
we have a canonical toric divisor: Let $P_k \subset P$ denote the face of $P$ defined by $\ell_k = 0$. Then $D_k = \mu^{-1}(P_k)$ is a smooth sphere and 
\[D = \bigcup_{k=1}^N D_k\]
is a normal crossing divisor. The cohomology class of this divisor depends on $n_k$ but not on $\lambda_k$. We consider the family $P_t$ of polytopes defined by 
\[\langle y, n_k\rangle  \le \lambda_k -t\]
for $t \in [0,T]$. This canonically defines a 1-parameter family of toric K\"{a}hler metrics $\widetilde{\omega}_t$ via the Guillemin potentials
\[\widetilde{u}_t = \frac{1}{2}\sum_{k=1}^N\ell_{k,t}(y)\log(\ell_{k,t}(y))\]
Note that since $\widetilde{\omega}_0$ and $\omega$ share a moment polytope, we have $[\omega] = [\widetilde{\omega}]$ by \cite{Abreu1998}.

We may present $[\widetilde{\omega}_t]$ and $c_1(X)$ via the toric divisors \cite{Guillemin1994} \cite[Exercise 3.5]{Abreu2003} \cite[Chapter 8]{CLS2011}, we have 
\[[\widetilde{\omega}_t] = 2\pi\sum_{k=1}^N (\lambda_k - t)[\omega_{D_k}]\]
and 
\[c_1(X) = \sum_{k=1}^N [\omega_{D_k}]\]
after choosing compactly supported 2-forms $\omega_{D_k}$ as before. This allows us to write
\[[\widetilde{\omega}_t] = [\omega] - 2\pi t c_1(X).\]
Thus our family follows the cohomological profile of a K\"{a}hler-Ricci flow starting at $\omega$.

\subsection{The geometry of the fiber}
The moment map $\mu \colon X \to P$ restricts to a fiber $F$ to yield a map 
$\mu|_{F}\ \colon F \to P$. Its image is a curve $\gamma_F \subset P$ with endpoints on $\partial P$. Let $v \in \mathfrak{t}$ be a vector generating a compatible Hamiltonian $S^1$-action on $F$. We can extend $v$ to an integral basis $\{v,w\}$ for the Lie algebra $\mathfrak{t}$ which we may dualize to obtain a pair of moment coordinates $(y_1,y_2)$ for $\mathfrak{t}^*$. This gives rise to a pair of toric coordinates $(z_1,z_2)$ in $X - D$.

The moment map $\mu_F \colon F \to \RR$ for its compatible Hamiltonian $S^1$-action is $\mu_F = \mu_1$. Its image is the interval $I_F \subset \RR$ given by $I_F = [s_-, s_+]$. Using this to fix action angle coordinates $(s,\varphi)$, we may write the fiber metric $g_F = g|_F$ in the form
\[g_F = S\ ds^2 + S^{-1}\ d\varphi^2\]
where $S = \Delta u_F$ for some convex toric potential $u_F \colon I_F \to \RR$ (c.f. \cite{Abreu1998} \cite[Section 2.2]{Abreu2003}).

We will compute $\Delta u_F$ from $\nabla^2 u$. Recall that for the Guillemin potential, we have
\[\nabla^2 u = \frac{1}{2}\sum_{k=1}^N\frac{d\ell_k\otimes d \ell_k}{\ell_k(y)}\]
We may use the moment coordinate $s$ to form a parameterization $\gamma_F(s) = (s,r(s))$ of the $\mu$-moment image of $F$. It follows that
\begin{align}\label{eqn:onshell}\nabla^2 u|_F = \frac{1}{2}\sum_{k=1}^N\frac{d\ell_k(\dot{\gamma}_F(s))^2}{\ell_k(\gamma_F(s))}.
\end{align}
Each $\ell_k$ is of the form
\[\ell_k(y) = \lambda_k - \langle n_k, y\rangle\]
for some primitive integral vector $n_k$ and so it may be written as
\[\ell_k(y_1,y_2) = \lambda_k + a_ky_1 + b_ky_2\]
where $a_k$ and $b_k$ are integers.

By letting $f_k(s) = \ell_k(\gamma_F(s))$, we may write \Cref{eqn:onshell} more compactly as 
\[\nabla^2 u|_F = \frac{1}{2}\sum_{k=1}^N\frac{(f'_k)^2}{f_k}\]
Since the fiber is given on the open torus orbit by $z_2=c$, via the fact that $z_j = e^{x_j + i\theta_j}$ we have
\[
\log|c| = x_2=\frac{\partial u}{\partial y_2}
\]
along $F$.  We therefore obtain
\[\log|c| = \frac{\partial u}{\partial y_2} = \frac{1}{2} \sum_{k=1}^N b_{k}(1 + \log \ell_k)\]
which exponentiates to 
\[\prod_{k=1}^N \ell_k^{b_k} = |c|^2e^{-\sum_{k=1}^N b_k}.\]
We let the right hand side be denoted by a constant $C = C(b_k,F) = |c|^2e^{-\sum_{k=1}^N b_k}$
and if we let $I_+ = \{k : b_k \ge 0\}$ and $I_- =  \{k : b_k < 0\}$, then
\[\prod_{k=1}^N \ell_k^{b_k} - C= 0\]
is equivalent to
\begin{align}\label{eqn:polynomial}
    \prod_{k_+ \in I_+} \ell_k^{b_{k_+}}- C\prod_{k_- \in I_-} \ell_k^{-b_{k_-}} = 0.
\end{align}
The left hand side of \Cref{eqn:polynomial} is \emph{polynomial} in $(y_1,y_2)$ and the moment curve $\gamma_F\subset P$ is the vanishing locus of this polynomial intersected with $P$. This fact will be used in the proof of \Cref{thm:main}.

\section{Kahler Ricci Flow on Toric Ruled Surfaces}
Let $(X,g,\omega)$ be a toric ruled surface. The fact that the fibers of such a manifold are Fano indicates that their volume should naturally shrink under the K\"{a}hler-Ricci flow. If the volume of the fiber collapses in finite time, then the cohomology class of the metric often collapses onto the class of a K\"{a}hler metric pulled back from the base. In this case, we shall refer to the fiber as collapsing. One can obtain a coarse understanding of collapsing behavior by analyzing volume.

In order to complete a proof of \Cref{thm:main}, we relate the diameter of the fiber to the volume of the fiber and show that for our canonical shrinking family of toric metrics, the control of diameter decay arises exactly from the collapse rate of the volume. This follows from a clear understanding of the toric fiber metric.

\subsection{Collapsing and non-collapsing}
We let 
\[\omega_{\Sigma} = \frac{f_*(\omega\wedge \omega)}{Vol(F,\omega)}\]
where $f_* \colon \Omega_{dR}^4(X) \to \Omega_{dR}^2(\Sigma)$ denotes fiber integration. We have
\[Vol(X,\omega) = \int_{X}\omega \wedge \omega = \int_{\Sigma}f_*(\omega\wedge \omega)\]
and so $Vol(X,\omega) = Vol(F,\omega)Vol(\Sigma,\omega_{\Sigma}).$

Under K\"{a}hler-Ricci flow, we have 
\begin{align*}
Vol(X,\omega_t) &= \int_{X} \omega_t \wedge \omega_t\\ 
&= \int_{X} (\omega - 2\pi t c_1(X)) \wedge (\omega - 2\pi t c_1(X))\\
&= \int_{X} \omega \wedge \omega - 4\pi t\int_{X} c_1(X) \wedge \omega + 4\pi^2 t^2 \int_{X} c_1(X) \wedge c_1(X) \\
&= Vol(X,\omega) - 4\pi t Vol(D,\omega) + 4\pi^2 t^2[D]\cdot[D]
\end{align*}
and of course
\begin{align*}
Vol(F,\omega_t) =  Vol(F,\omega) - 2\pi t[F]\cdot [D]
\end{align*}
A basic understanding of collapsing behavior of $X$ under K\"{a}hler-Ricci flow comes first viewing $Vol(X,\omega_t)$ as a quadratic in $t$ and understanding its discriminant. If 
\[Vol(X,\omega) > \frac{Vol(D,\omega)^2}{[D]\cdot[D]}\]
then total volume remains non-collapsed for all $t < T$ which means so does the volume of $F$. K\"{a}hler-Ricci flow thus continues until some $t = T$ with 
\[T \le T_F\]
and 
\[T \le T_D.\]

If 
\[Vol(X,\omega) \le \frac{Vol(D,\omega)^2}{[D]\cdot[D]}\]
then $X$ may collapse in several ways. We let $t = T_X$ denote the first positive root of $Vol(X,\omega_t)$. If strict inequality holds, then we have $T_X < T_D$ and so the volume of $X$ collapses before the volume of $D$ does. If $T_X < T_F$, then we have 
\[\lim_{t \to T_X} Vol(\Sigma, \omega_{\Sigma}(t)) = 0\]
while the volume of the fiber remains positive. If $T_X = T_F$, then by L'H\^opital's rule
\[\lim_{t \to T_X}Vol(\Sigma,\omega_\Sigma(t)) = \lim_{t \to T_X} \frac{Vol(X,\omega_t)}{Vol(F,\omega_t)} = \lim_{t \to T_X} \frac{\partial_t Vol(X,\omega_t)}{\partial_t Vol(F,\omega_t)} = \lim_{t \to T_{X}} \frac{2Vol(D,\omega_t)}{[F]\cdot[D]}\]
and so $(\Sigma, \omega_{\Sigma}(t))$ remains non-collapsed unless $T_X = T_F = T_D$. 

We will assume that the maximal existence time satisfies 
\[T = T_F < T_D\]
so that the K\"{a}hler-Ricci flow exists until volume of the fiber collapses and 
\[\lim_{t \to T_{X}} Vol(\Sigma,\omega_t) > 0.\]
this is exactly the same setting as \cite{SongSzekelyhidiWeinkove2013} where it is required that 
\[[\omega] - 2\pi Tc_1(X) = [f^*\Omega]\] 
for some K\"{a}hler metric $\Omega$ on $\Sigma$.

\subsection{Diameter decay under K\"{a}hler-Ricci flow}\label{sec:diameter}
We are now ready to prove our main theorem. The diameter of $F$ is defined to be the supremum over all pairs of points $p,q \in F$ of $g_F$-distances $dist_{g_F}(p,q)$. Let $p,q$ be a pair of points realizing the diameter. We may always join any two points in $F$ by segments of a pair of meridians through the poles of $F$. This allows us to approximate the distance between $p$ and $q$ with curve with piecewise constant angle $\varphi$ in action angle coordinates $(s,\varphi)$. We thus may estimate the diameter to be
\[diam_{g_F}(F) \le 2L[F],\]
where $L[F]$ is the meridional length of $F$ defined to be the length of any fixed meridian $m = \{\varphi = \varphi_0\}$ between the poles $p_- = \{s = s_{-}\}$ and $p_+ = \{s = s_+\}$. We have
\[L[F] = Length[m] = \int_{I_F} \sqrt{S}\ ds.\]

From our formula for $S$ above, 
\[S = \frac{1}{2}\sum_{k=1}^N\frac{(f'_k)^2}{f_k}\]
this becomes the quantity
\[L[F] \le \frac{1}{\sqrt{2}}\sum_{k=1}^N\int_{I_F}\frac{|f'_k|}{\sqrt{f_k}}\ ds.\]
Each integral on the right may be computed via the Jacobian formula which yields
\[\int_{I_F}\frac{|f'_k|}{\sqrt{f_k}}\ ds = \int_{\RR} \sum_{s \in f^{-1}_k(t)}\frac{1}{\sqrt{t}}\ dt = \int_{f_{k,\min}}^{f_{k,\max}} \frac{N_k(t)}{\sqrt{t}}\ dt\]
where 
\[f_{k,\min} = \min_{s} f_k(s),\]
\[f_{k,\max} = \max_{s} f_k(s),\]
and
\[N_k(t) = \#f^{-1}_k(t) = \# (\gamma_F \cap \{\ell_k = t\}).\]
By \Cref{eqn:polynomial} and the fact that $\ell_k$ is also polynomial, the right hand side is finite and uniformly bounded above by Bezout's theorem. We denote the upper bound by $K$. Note that this bound depends only on the fiber $F$ and the integers $b_k$.

It is straightforward to show $|f'_k|$ is bounded above.
Indeed, by the parameterization of $\gamma_F(s)$, 
\[f_k'(s) = a_k + b_kr'(s).\]
The vector field $\dot \gamma_F(s) = (1, r'(s))$ is tangent to the level set $F = \{x_2 = \log|c|\}$ and so $\dot \gamma_F \in \ker(dx_2) = TF$ and so
\[0 = dx_2(\dot \gamma_F) = d \left(\frac{\partial u}{\partial y_2}\right) (\dot\gamma_F) = u_{12} + u_{22}r'(s)\]
which yields
\[r'(s) = -\frac{u_{12}}{u_{22}}.\]
Since
\[u_{12} = \frac{1}{2}\sum_{k=1}^N \frac{a_kb_k}{\ell_k}\]
\[u_{22} = \frac{1}{2}\sum_{k=1}^N \frac{b_k^2}{\ell_k}\]
we have 
\[r'(s) = -\frac{\sum_{k=1}^N \frac{a_kb_k}{\ell_k}}{\sum_{k=1}^N \frac{b_k^2}{\ell_k}}\]
and since 
\[\sum_{k=1}^N \frac{|a_kb_k|}{\ell_k}\le \left(\max_{k, b_k \neq 0} \left|\frac{a_k}{b_k}\right|\right)\sum_{k=1}^N \frac{b_k^2}{\ell_k}\]
we obtain upper bounds 
\[|r'(s)| \le \max_{k, b_k \neq 0} \left|\frac{a_k}{b_k}\right|\]
\[|f'_k(s)| \le |a_k| + |b_k|\max_{k, b_k \neq 0}\left|\frac{a_k}{b_k}\right|. \]
Thus, there exists a constant $C = C(F,a_1,b_1,\ldots, a_k,b_k)$ depending only on $F$ and the fan of $\{\ell_k\}_{k=1}^N$ and uniform in $k$ such that $|f_k'| \le C$. It then follows that 
\[\sqrt{f_{k,\max} - f_{k,\min}} \le C\sqrt{s_{+} - s_{-}}.\]

We have an upper bound
\begin{align*}
    L[F] &\le \frac{1}{\sqrt{2}}\sum_{k=1}^n\int_{f_{k,\min}}^{f_{k,\max}} \frac{K}{\sqrt{t}}\ dt\\ &= \sqrt{2}K\sum_{i=1}^N(\sqrt{f_{k,\max}} - \sqrt{f_{k,\min}})\\ &\le \sqrt{2}K\sum_{k=1}^N\sqrt{f_{k,\max} - f_{k,\min}}\\ &\le \sqrt{2}NKC \sqrt{(s_+ - s_{-})}
\end{align*}
Since 
\[Vol(F,\omega) = \int_{F} ds \wedge d\varphi = 2\pi(s_+ - s_-)\] 
this yields the estimate
\[diam_{g_F}(F) \le C\sqrt{Vol(F,\omega)}.\]
for some constant $C = C(F, a_k,b_k)$ depending only on the fan of $\{\ell_k\}_{k=1}^N$ and $F$. From this inequality follows our main theorem.

\begin{proof}[\normalfont \textbf{Theorem~\ref{thm:main}}]
\renewcommand{\qedsymbol}{}
\em
    Let $\widetilde{\omega}_t$ be the canonical shrinking family of toric metrics starting at $(X,g,\omega)$. Then we have
    \[diam_{\widetilde{\omega}_t}(F) \le C(T - t)^{1/2}\]
    where $C$ is a constant depending only on the fiber $F$ and the fan of the linear functions $\{\ell_k\}_{k=1}^N$.
\end{proof}
\begin{proof}
    Along the family, we have 
    \[diam_{\widetilde{\omega}_{t}}(F) \le C\sqrt{Vol(F,\widetilde{\omega}_t)}\]
    Since $\widetilde{\omega}_t$ and $\omega_t$ are cohomologous,
    \[Vol(F,\widetilde{\omega}_t) = Vol(F,\omega_t) = Vol(F,\omega) - 2\pi t[F]\cdot[D] = 2\pi[F]\cdot[D]\left(\frac{Vol(F,\omega)}{2\pi[F]\cdot[D]} - t\right)\]
    and since
    \[T = T_{F} = \frac{Vol(F,\omega)}{2\pi[F]\cdot[D]}\]
    we have 
    \[Vol(F,\widetilde{\omega}_t) = 2\pi[F]\cdot[D]\left(T - t\right).\]
    We may then estimate
    \[diam_{\widetilde{\omega}_t}(F) \le C\sqrt{2\pi[F]\cdot[D]}(T - t)^{1/2}\]

    We replace $C$ in the theorem statement with the constant on the right which depends only on $F$ and the fan of $\{\ell_k\}_{k=1}^N$.
\end{proof}

\section{Discussion}
We now have seen that every toric variety $(X,g,\omega)$ admits a canonical shrinking family of toric metrics $\widetilde{\omega_{t}}$ which follow the cohomological profile of a K\"{a}hler-Ricci flow starting at $\omega$. Analytically, we have
\[\omega_t = i \partial \overline{\partial} \phi_t = \widetilde{\omega}_t + i\partial \overline{\partial}\psi_t\]
With $\psi_t$ satisfying the classic parabolic Monge-Ampere evolution equation under the flow. This is purely a manifestation of the framework of Guillemin and Abreau which ultimately shows that any two toric metrics are $i\partial \overline{\partial}$-exact. The key estimate in \Cref{sec:diameter} relies on the understanding and control of the toric fiber metrics $g_F(t)$ for the family $\widetilde{\omega}_t$ which relies on the specific toric potential $u$ chosen from \Cref{eqn:guillemin}. By \cite{Abreu1998} \cite[Theorem 2.8]{Abreu2003}, any toric potential $u_{\omega}$ for our original toric metric $\omega$ differs from $u$ by a single smooth function $f \colon P \to \RR$ 
\[u_\omega = u + f.\]
The toric metric $g$ is then obtained via the Hessian of $u_\omega$. If $\widetilde{\phi}_t$ is a family of K\"{a}hler potentials for $\widetilde{\omega_t}$, we have that $\phi = \widetilde{\phi}_t + \psi_t$ is a family of K\"{a}hler potentials for $\omega_t$ which are related to toric potentials $u_{\omega_t}$ via Legendre transform. Again, we have the difference:
\[u_{\omega_t} = u_t + f_t.\]
It follows that an understanding of the $C^2$-regularity of the associated smooth functions $f_t$ is a direct obstruction to transferring the estimate from \Cref{thm:main} to the Kähler-Ricci flow itself. In particular, uniform control of the second derivatives of $\psi_t$ along the collapsing fibers would provide the natural bridge between the symplectic model and the sharp Kähler-Ricci flow estimate.

\newpage
%%%%%%%%%%%%%%%%%%%%%%%%%%%%%%%%%%%%%%%%%%%%%%%%%%%
\bibliographystyle{plain}
\bibliography{main}
%%%%%%%%%%%%%%%%%%%%%%%%%%%%%%%%%%%%%%%%%%%%%%%%%%%
\end{document}